\documentclass[11pt]{article}
\usepackage{amsfonts}
\usepackage{mathrsfs,color,amscd,amssymb,enumerate,amsthm,amsmath,bm,graphicx,psfrag,subfigure}
\usepackage{latexsym}
\usepackage{float,fancybox,shapepar,setspace,hyperref}
\usepackage{pgf,tikz}
\usetikzlibrary{shapes.geometric}

\makeatletter
\def\leftharpoonfill@{\arrowfill@\leftharpoonup\relbar\relbar}
\def\rightharpoonfill@{\arrowfill@\relbar\relbar\rightharpoonup}
\newcommand\rbjt{\mathpalette{\overarrow@\rightharpoonfill@}}
\newcommand\lbjt{\mathpalette{\overarrow@\leftharpoonfill@}}
\makeatother

\makeatletter

\renewcommand{\@seccntformat}[1]{{\csname the#1\endcsname}{\normalsize .}\hspace{.5em}}
\makeatother

\def \ss {\subseteq}

\def\bqed{ \hfill $\blacksquare$}

\newtheorem{thm}{Theorem}

\newtheorem{claim}{Claim}

\newtheorem{lem}{Lemma}

\newtheorem*{claim*}{Claim}

\theoremstyle{definition}

\newtheorem*{qu*}{Question}

\usetikzlibrary{arrows}
\usetikzlibrary{calc}
\begin{document}

\title{New results on proper orientation number of graphs}
\author{Lanchao Wang\thanks{School of Mathematics, Nanjing University, Nanjing, China, and ECOPRO, Institute for Basic Science, 55 Expo-ro, Yuseong-gu, Daejeon, 44126, Korea (Email: lanchaowang@foxmail.com)}~, Xiaolin Wang\thanks{ Corresponding author. School of Mathematics and Statistics, Fuzhou University, Fuzhou
350108, P.R. China (Email: xiaolinw@fzu.edu.cn)}~, Guangmiao Yu\thanks{School of Mathematics and Statistics, Fuzhou University, Fuzhou
350108, P.R. China (Email: guangmiaoyu@163.com)}}

\date{}
\maketitle

\begin{abstract}
The proper orientation number $\vec{\chi}(G)$ of an undirected graph $G$ is the 
minimum $k$ such that there exists an orientation of $G$ with all out-degrees at most $k$ and with different out-degrees for any two adjacent vertices. 
Chen, Mohar and Wu (JCTB, 2023)  proved that 
if $G$ is a $r$-partite graph, then
$\vec{\chi}(G) \leq \frac{1}{2} \text{Mad}(G)+r^{1+o(1)}$, where $\text{Mad}(G)$ is  the maximum average degree of $G$. Moreover, if $G$ is a bipartite graph, then $ \vec{\chi}(G) \leq \lceil \frac{1}{2} \text{Mad}(G)\rceil +3$ and this bound is tight.  They also asked whether $\vec{\chi}(G)-\lceil \frac{1}{2} \text{Mad}(G)\rceil$ can be bounded by  a linear function of $r$.  

In this paper, we first  construct  somewhat involved $r$-partite graphs with $\vec{\chi}(G)\geq\lceil \frac{1}{2} \text{Mad}(G)\rceil +\lfloor\frac{5}{2}r\rfloor-2$, showing that a linear dependence on \(r\) is unavoidable. We also prove that   $ \vec{\chi}(G) \leq\lceil \frac{1}{2} \text{Mad}(G)\rceil +7$ for every 3-partite graph $G$. This implies \(\vec{\chi}(G)\le 10\) for \(3\)-colorable planar graphs and
\(\vec{\chi}(G)\le 9\) for outerplanar graphs, improving the corresponding bounds
of Chen, Mohar, and Wu.

\vskip 2mm

\end{abstract}
\setcounter{section}{0}

\section{Introduction}\setcounter{equation}{0}

\vskip 2mm

 Let $G$ be a simple undirected connected graph. 
A \emph{proper orientation} $D$ of $G$ is an orientation of $G$ such that  $d_D^+(u)\not=d_D^+(v)$ for any edge $uv\in E(G)$.
Denote the \emph{proper orientation number} $\vec{\chi}(G)$ of  $G$ as the
minimum $\Delta^+(D)$ among all proper orientations $D$ of $G$, where $\Delta^+(D)$ is the maximum out-degree of $D$.

The study of proper orientations was motivated by the 1-2-3 Conjecture of Karoński,  Łuczak and  Thomason \cite{KT} (confirmed by Keusch in 2024 \cite{K}), which states that for every graph without isolated edges, its edges can be assigned weights from $\{1,2,3\}$ so that any two adjacent vertices receive different sums of incident edge weights.   
The proper orientation number can be viewed as the directed version of the 1-2-3 Conjecture: if  $uv$ is oriented as $(u,v)$, then $uv$ is assigned weight 1, but only $u$ can receive this weight.

In 2013,
Ahadi and Dehghan \cite{AD} initiated the systematic study of the proper orientation number. 
From the viewpoint of computational complexity, deciding whether $\vec{\chi}(G)\leq k$ for a graph $G$ and a positive integer $k$ is NP-complete, even for the following special classes of graphs: the line graphs of  regular graphs \cite{AD}, the planar bipartite graphs with maximum degree five \cite{ACD} and the planar subcubic graphs  \cite{ACD}. These results show that any insight into  the study of proper orientation number is  interesting.

A proper  $k$-coloring of $G$ is a function $c: V(G)\rightarrow \{1,...,k\}$ such that for any two adjacent vertices $u,v\in V(G)$, $c(u)\neq c(v)$. The smallest integer $k$ such that $G$ has a proper  $k$-coloring is the \emph{chromatic number} of $G$,  and is denoted by $\chi(G)$. Obviously, every proper orientation of $G$ induces a proper coloring on $V(G)$. This implies that $\chi(G)\leq \vec{\chi}(G)+1$.
Hence, it is   natural
to ask how $\vec{\chi}(G)$ behaves on graphs with bounded $\chi(G)$, or on
$r$-partite graphs.

It is easy to see that $ \vec{\chi}(G)\leq \Delta (G)$, where $\Delta (G)$ is the maximum degree of $G$. In 2015, 
Araujo, Cohen, de Rezende, Havet and Moura \cite{ACD} proved that the proper orientation number of bipartite graphs satisfies $\vec{\chi}(G)\leq \ \lfloor \frac{1}{2}(\Delta(G)+\sqrt{\Delta(G)}) \rfloor +1$. They asked whether there exists a constant $C$ such that $\vec{\chi}(G)\leq \  \frac{1}{2}\Delta(G)+C$ for every bipartite graph.

A breakthrough came from Chen, Mohar and Wu \cite{CM} in 2023.
They proved an essentially best possible general upper bound.
The \emph{maximum average degree} of a graph $G$, denoted by $\text{Mad}(G)$, is the largest average degree over all subgraphs of $G$:

 $$\text{Mad}(G)=\max\limits_{H\ss G} \frac{2|E(H)|}{|V(H)|}.$$

Observe that for any orientation $D$ of $G$,
\(\Delta^+(D)\geq \frac{1}{2}\text{Mad}(G)\), since every subgraph \(H\subseteq G\) satisfies 
\(|E(H)|\le \Delta^+(D)|V(H)|\). Hence, $\vec{\chi}(G)\geq \lceil\frac{1}{2}\text{Mad}(G)\rceil$. And this bound is sharp, as shown by  $K_{{2k-1},{2k-1}}$. 
Chen, Mohar and Wu \cite{CM} proved the following theorem.

\begin{thm}\label{CMW1} (Chen, Mohar and Wu \cite{CM})
If $G$ is a bipartite graph, then  $\vec{\chi}(G) \leq \lceil \frac{1}{2} \text{Mad}(G)\rceil +3$, and this bound is tight.
\end{thm}

As a corollary, $\vec{\chi}(G) \leq \lceil \frac{1}{2} \Delta(G)\rceil +3$, answering the problem of Araujo, Cohen, de Rezende, Havet and Moura \cite{ACD}.
Chen, Mohar and Wu \cite{CM} also proved a result for 
$r$-partite graphs.

\begin{thm}\label{CMW2} (Chen, Mohar and Wu \cite{CM})
    If $G$ is a $r$-partite graph, then $\vec{\chi}(G) \leq \frac{1}{2} \text{Mad}(G)+ O(\frac{r\log{r}}{\log{\log{r}}})$. 
\end{thm}

The proofs of Theorems \ref{CMW1} and \ref{CMW2} are based on two novel notions of potential out-degree and fractional orientations, and a clever use of weighted  matching Lemma \ref{l2}.  

It is also interesting to investigate the tight upper bound of proper orientation number for 
 sparse  graphs. For example, earlier studies   \cite{AG}, \cite{AC},  \cite{ACD},  \cite{AH}, \cite{KG} and  \cite{KN} focused on chordal graphs, trees, outerplanar and bipartite planar graphs. These works established that the proper orientation number can be  bounded, but with  some supplementary structural constraints:  triangle-free, bipartite, or  3-connected. 
As a corollary of Theorem \ref{CMW1}, if $G$ is a bipartite planar graph, then $\vec{\chi}(G)\leq 5$, and this bound is tight by the same extremal graph. Chen, Mohar and Wu \cite{CM} also  use the same methods developed for Theorems \ref{CMW1} and \ref{CMW2} to prove the following results.

\begin{thm}
    \label{CMW3} (Chen, Mohar and Wu \cite{CM})
If $G$ is a planar graph, then $\vec{\chi}(G)\leq 14$. Moreover, if $G$ is 3-colorable, then $\vec{\chi}(G)\leq 11$; if $G$ is outerplanar, then $\vec{\chi}(G)\leq 10$.
\end{thm}

These bounds in Theorem \ref{CMW3} may  not be tight, as 
a planar graph with $\vec{\chi}(G)=10$ and an outerplanar graph with $\vec{\chi}(G)=7$ were constructed by 
Araujo,  Havet,  Linhares Sales and  Silva \cite{AH}.

In this paper, our first main result improves the bound of Chen, Mohar and Wu in Theorem~\ref{CMW2} for
3-partite graphs, using the notion of 
potential out-degree and weighted  matching Lemma \ref{l2}, but without using  fractional orientations.

\begin{thm}
    \label{t1} For every 3-partite graph $G$,  $\vec{\chi}(G)\leq \left\lceil \frac{\text{Mad}(G)}{2} \right\rceil+7$.
    
\end{thm}

Since every $3$-colorable planar graph satisfies $\text{Mad}(G) < 6$, and every outerplanar graph
satisfies $\text{Mad}(G) < 4$, Theorem \ref{t1} immediately yields the following improvement of
Theorem \ref{CMW3}.

\begin{thm}
    \label{planar} Let $G$ be a 3-colorable planar graph. Then  $\vec{\chi}(G)\leq 10$.
    Moreover, if $G$ is outerplanar, then $\vec{\chi}(G)\leq 9$.
\end{thm}

Chen, Mohar and Wu \cite{CM}  also asked (Question 7.3) whether $\vec{\chi}(G)- \left\lceil \frac{\text{Mad}(G)}{2} \right\rceil$ can be bounded by  a linear function of $\chi(G)$. In this paper, we 
construct somewhat involved   $r$-partite graphs with $\vec{\chi}(G)\geq \left\lceil \frac{\text{Mad}(G)}{2} \right\rceil+\left\lfloor \frac{5}{2}r\right\rfloor-2$. Our construction shows that in any bound of the form
\(\vec{\chi}(G)\le \left\lceil \frac{\text{Mad}(G)}{2}\right\rceil+f(r),
\)
the term \(f(r)\) must be asymptotically at least \(5r/2\).

  In the rest of this paper, we first construct a class of $r$-partite graphs with $\vec{\chi}(G)\geq \left\lceil \frac{\text{Mad}(G)}{2} \right\rceil+\left\lfloor \frac{5}{2}r\right\rfloor-2$ in Section 2. In Section 3, we prove Theorem \ref{t1}. 
  
  At the end of this section, we  present some necessary notation. Let $[a, b]$ denote the set of all integers between $a$ and $b$. We say a graph is \emph{$k$-degenerate} if all its subgraphs have minimum degree at most $k$. Let $A$ be an independent set in $G$.   To   \emph{attach} a graph $F$ to $A$  means that we obtain a new graph $G\cup F$, while $V(G)\cap V(F)=A$.   For any two disjoint vertex sets $A,B\ss V(G)$,  $A\to B$ means that we orient all the edges in $E[A,B]$  from $A$ to $B$. If $A=\{a\}$ or $B=\{b\}$, we simply denote by $a\to B$ or $A\to b$. Denote the \emph{out-edge} (\emph{in-edge}) incident with $v$ as the edge oriented out of (into) $v$. We refer the reader to \cite{BM} for other definitions not included here.

\section{$r$-partite graphs $G$ with large $\vec{\chi}(G)$ ($r\geq 3$)}

\noindent{\bf Sketch.}
The construction in this section is quite involved.  To prove the lower bound,
we argue by contradiction and assume that there is a proper orientation
with maximum out-degree at most \(M\).  Under this assumption, we control
the possible out-degree values on a basic independent set \(A\).  The
construction has three main ingredients.

\smallskip
\noindent{\rm (1)}
We introduce forcers, whose role is to forbid prescribed out-degree
values on \(A\) or on suitable subsets of \(A\).   The bipartite forcer
comes from the construction of Chen, Mohar and Wu~\cite{CM}; we then
construct an \(r\)-partite interval forcer to forbid further consecutive
values. Although $\vec{\chi}(G)$ is not monotonic \cite{AH}, we can combine all the forcers
together to forbid a larger out-degree interval on $A$.

\smallskip
\noindent{\rm (2)}
We control \(\text{Mad}(G)\).  Since the new forcers are not
necessarily \(k\)-degenerate, we use Lemma~\ref{lemma:mad-charging} to
control the maximum average degree through an auxiliary orientation.

\smallskip
\noindent{\rm (3)}
We take \(r\) copies of the block and join the selected controlled sets
into a complete \(r\)-partite subgraph.  The forcers leave only a small
set of possible out-degree values on these selected vertices.  This
forces a rigid count of the total out-degree on the joined part, which
is too small to cover all its edges.  This gives the desired
contradiction.

\subsection{Forcers.}
 Let \(A\) be a fixed independent 
set in some graph $G$ and let \(I\) be a set of positive integers. A graph \(F\) with
 \(A\subseteq V(F)\) is called an \(I\)-forcer for \(A\) under the  upper bound \(M\)
if the following holds: when we attach \(F\)  to $A$,
every proper orientation  of \(G\cup F\) with
\(
        d^+(v)\leq M\)
        for every \(v\in V(G\cup F)
\)
must satisfy
\[
        d^+(a)\notin I
        \quad\text{for every }a\in A.
\]
In other words, an \(I\)-forcer does not prescribe the out-degrees of
vertices in \(A\); rather, it forbids the values in \(I\) on \(A\). This definition is inspired by the bipartite forcer constructed by Chen, Mohar and Wu~\cite{CM} in their proof of the tightness of Theorem~\ref{CMW1}. We include a proof here for completeness. 

\begin{lem}[{\cite{CM}}]\label{lemma:low-forcer}
Let \(A\) be an independent set of size \(k\) in some graph $G$, and let \(M\geq k+1\). There is a
bipartite graph \(F_1\) with bipartition
\(
        \{A\cup D,\ B\cup C\}
\)
such that \(F_1\) is a \([1,k+1]\)-forcer for \(A\) with
upper bound \(M\).
\end{lem}
\begin{proof}
We construct \(B,C,D\) and edges in $F_1$ as follows; a schematic illustration is given in Figure~\ref{fig:F1-forcer}. For every 
\(\emptyset\neq S\subseteq A\), add a set \(B_S\) of \(kM+1\) vertices, each of which
is adjacent exactly to all vertices of \(S\). Let
\(B=\bigcup_{\emptyset\neq S\subseteq A} B_S .\)

Next, take \(kM+1\) pairwise disjoint \(k\)-sets
\(C_1,C_2,\ldots,C_{kM+1}.
\)
All vertices of every \(C_j\) are adjacent to all vertices of \(A\). For
each \(j\), add one vertex \(d_j\) adjacent exactly to all \(k\) vertices
of \(C_j\). Let
\[
        C=C_1\cup C_2\cup\cdots\cup C_{kM+1}
        \quad\text{and}\quad
        D=\{d_1,d_2,\ldots,d_{kM+1}\}.
\]
This defines the bipartite graph \(F_1\) with bipartition
\(\{A\cup D,\ B\cup C\}.
\)

We now prove that \(F_1\) is a \([1,k+1]\)-forcer for \(A\) with upper bound $M$.
 Let \(p\) be any
proper orientation of \(G\cup F_1\) satisfying
\(d^+(v)\leq M\) for every \(v\in V(G\cup F_1).
\)

First we show that no vertex of \(A\) has out-degree in
\([1,k]\).   For any \(a\in A\) and any \(1\leq i\leq k\), choose an
\(i\)-subset \(S\subseteq A\) with \(a\in S\). Note that every vertex in \(B_S\)
has degree \(i\). If every vertex in \(B_S\) has out-degree at most \(i-1\), then each vertex in \(B_S\) receives at least one edge from \(S\).
Hence, at least
\(|B_S|=kM+1\)
edges are oriented from \(S\) to \(B_S\). But  the number of edges oriented
out of \(S\) is at most
\(iM< kM+1,\)
a contradiction. Therefore some vertex of \(B_S\), which is adjacent to \(a\),
has out-degree \(i\). Hence 
\(d^+(a)\neq i\) for any $a\in A$ and any $1\leq i\leq k$.

It remains to forbid the value \(k+1\). Since \(|A|=k\), at most \(kM\)
edges can be oriented from \(A\) to \(C\). However, \(C\) is the union of
\(kM+1\) disjoint \(k\)-sets \(C_j\). 
There must exist some \(C_j\) such that $ C_j\to A$, and hence
\(d^+(c)\geq k\) for any $c\in C_j$.
Note that $d(c)=k+1$.
If every vertex of \(C_j\) had out-degree exactly \(k\), then  \(d_j\to C_j\).
But then
\(d^+(d_j)=k,
\) contradicting properness.
 Therefore some vertex \(c\in C_j\) has out-degree \(k+1\).
 Hence, by properness,
\( d^+(a)\neq k+1
\) for every $a\in A$.
\end{proof}\begin{figure}[htbp]
\centering
\begin{tikzpicture}[
    x=1cm,y=1cm,
    every node/.style={font=\small},
    line width=0.9pt,
    line cap=round,
    line join=round,
    edge/.style={black, line width=0.55pt}
]

\node[draw,ellipse,minimum width=2.2cm,minimum height=2.2cm] (A) at (-3.2,1.8) {};
\node at (-3.2,2.65) {\(\,A\,\)};
\node at (-3.2,2.33) {size \(k\)};

\node[draw,ellipse,minimum width=1.2cm,minimum height=1.2cm] (S) at (-3.1,1.45) {\(S\)};

\node[draw,ellipse,minimum width=6.2cm,minimum height=4.1cm] (B) at (3,2.0) {};
\node at (3.4,3.5) {\(B=\bigcup_{\emptyset \neq S\subseteq A}B_S\)};

\node[draw,ellipse,minimum width=3cm,minimum height=3cm] (BS) at (3.0,1.6) {};
\node at (3.0,2.5) {\(B_S\)};
\node at (3.0,2.05) {size \(kM+1\)};

\node[draw,ellipse,minimum width=5.0cm,minimum height=3.3cm] (C) at (2.9,-2.0) {};
\node at (2.9,-0.8) {\(C=\bigcup_{j=1}^{kM+1}C_j\)};

\node[draw,ellipse,minimum width=2.2cm,minimum height=2.2cm] (Cj) at (2.6,-2.35) {};
\node at (2.6,-1.7) {\(C_j\)};
\node at (2.6,-2.1) {size \(k\)};

\node[draw,ellipse,minimum width=3cm,minimum height=3cm] (D) at (-3.0,-1.85) {};
\node at (-3.0,-0.75) {\(D\)};
\node at (-3.0,-1.15) {size \(kM+1\)};

\coordinate (dj) at (-2.25,-2.0);
\fill (dj) circle (1.5pt);
\node[left=3pt] at (dj) {\(d_j\)};

\draw[edge]
    ($(S.east)+(-0.4,0.57)$)
    to[out=8,in=178]
    ($(BS.west)+(0.25,0.82)$);

\draw[edge]
    ($(S.east)+(-0.11,0.32)$)
    to[out=3,in=180]
    ($(BS.west)+(0.07,0.36)$);

\draw[edge]
    ($(S.east)+(0,0.05)$)
    to[out=0,in=182]
    ($(BS.west)+(0,-0.16)$);

\draw[edge]
    ($(A.east)+(-0.2,-0.6)$)
    to[out=-30,in=154]
    ($(Cj.west)+(0.43,0.88)$);

\draw[edge]
    ($(A.east)+(-0.4,-0.85)$)
    to[out=-30,in=157]
    ($(Cj.west)+(0.24,0.65)$);

\draw[edge]
    ($(A.east)+(-0.7,-1.02)$)
    to[out=-30,in=162]
    ($(Cj.west)+(0.1,0.4)$);

\draw[edge]
    ($(dj)+(0.04,0.02)$)
    to[out=3,in=180]
    ($(Cj.west)+(0,0.08)$);

\draw[edge]
    ($(dj)+(0.04,0)$)
    to[out=-4,in=180]
    ($(Cj.west)+(0.05,-0.22)$);

\draw[edge]
    ($(dj)+(0.04,-0.02)$)
    to[out=-11,in=180]
    ($(Cj.west)+(0.13,-0.50)$);

\end{tikzpicture}
\caption{The \(\{1,\ldots,k+1\}\)-forcer \(F_1\).}
\label{fig:F1-forcer}
\end{figure}

Chen, Mohar and Wu~\cite{CM} use this bipartite
 forcer \(F_1\)
 to construct the extremal graphs for Theorem \ref{CMW1}. In order to obtain 
 the extremal graphs for $r$-partite graphs, we construct the following  $r$-partite forcer $F_2$, 
which is built from a new mechanism, as shown in Figure \ref{fig:F2-forcer}.
\begin{lem}\label{lemma:high-forcer}
Let \(r\geq 3\), and let \(a,b\) be integers such that
\(1\le a\leq b\) and \(
b-a+1\leq r-1.\)
Assume that
\(k'=k+a-r+1\ge 2\). Let \(A\) be an independent set of \(k'\) vertices in some graph $G$, and let
\(M\geq k+b\). There is an \(r\)-partite    \([k+a,k+b]\)-forcer \(F_2\) for \(A\) with upper
bound \(M\).
\end{lem}

\begin{proof}
 Let
\(Y_0,Y_1,\ldots,Y_{r-1}
\)
be pairwise disjoint  \((k'-2)\)-sets.
First, we construct one forcing block \(P^*\). For every
\(1\leq i\leq b-a+1\), let \(Q_i\) be obtained from a complete graph on
vertices
\(\{x^i,y_0^i,y_1^i,\ldots,y_{r-1}^i\}
\)
by deleting the edge \(x^iy_0^i\). Let \(P^*\) be the graph obtained
from the disjoint union of
\(Q_1,Q_2,\ldots,Q_{b-a+1}
\)
by adding all edges \(x^ix^j\) for \(1\leq i<j\leq b-a+1\).
Take
\(T=M\bigl(k'+r(k'-2)\bigr)+1
\)
copies of \(P^*\), denoted by
\(P_1,P_2,\ldots,P_T.
\)
Let
\(P=P_1\cup P_2\cup\cdots\cup P_T.\)
Now join every copy of \(x^i\) to all vertices of \(A\), and join every copy of
\(y_j^i\) to all vertices of \(Y_j\). 
Then we get the resulting graph \(F_2\)  with vertex set
\(A\cup V(P)\cup Y_0\cup\cdots\cup Y_{r-1}.
\)

We first indicate that \(F_2\) is \(r\)-partite. Put the vertices of
\(A\) and \(Y_0\) in one color class, say \(V_0\). For each
\(1\leq i\leq b-a+1\), put \(x^i\) and \(y_0^i\) in \(V_i\). For
\(1\leq j\leq r-1\), put \(Y_j\) in \(V_j\), and put \(y_j^i\) in
\(V_{i+j}\), where the subscripts are taken modulo \(r\). Since $b-a+1\leq r-1$, one can check that this gives a proper coloring. Moreover, the clique on vertices $\{y_0^i,\ldots,y_{r-1}^i\}$ shows that $\chi(F_2)=r$.

Now we show that the graph \(F_2\) is a \([k+a,k+b]\)-forcer for \(A\) with upper bound $M$.   Let \(p\) be
any proper orientation of \(G\cup F_2\) satisfying
\(d^+(v)\leq M
\) for every \(v\in V(G\cup F_2).
\)
Since
$T=M\bigl(k'+r(k'-2)\bigr)+1$, $|A|=k'$ and $\sum_{j=0}^{r-1}|Y_j|=r(k'-2)$,
there must exist a  \(P_t\subseteq P\) such that 
\(P_t\to 
        A\cup Y_0\cup\cdots\cup Y_{r-1}
\). Now we focus on such $P_t$. Then in  $P_t$, $x^i\to A$ and $y_j^i\to Y_j$ for any $i,j$.
Let
\(X=\{x^1,x^2,\ldots,x^{b-a+1}\}
\). For any  \(x^i\in X\),
\begin{align}\label{dxq}
d^+(x^i)=k'+d_X^+(x^i)+d_{Q_i}^+(x^i),
\end{align}
For every \(0\leq j\leq r-1\),
\begin{align}\label{dyq}
d^+(y_j^i)
=k'-2+d_{Q_i}^+(y_j^i).
\end{align}

\begin{claim}\label{X+Q}
 \(
        d_X^+(x^i)+d_{Q_i}^+(x^i)\geq r-1.
\)
\end{claim}
\begin{proof}
    
Suppose to the contrary that  \(
        d_X^+(x^i)+d_{Q_i}^+(x^i)\le r-2.
\)
Note that the vertices
\(y_0^i,y_1^i,\ldots,y_{r-1}^i
\)
form a clique. By (\ref{dyq}), the values \(d_{Q_i}^+(y_j^i)\), \(0\le j\le r-1\), are pairwise distinct.
By counting the edges of \(Q_i\) in two ways, we have $\sum_{j=0}^{r-1}d_{Q_i}^{+}(y_j^i)=\binom{r}{2}+(r-1)-d_{Q_i}^{+}(x^i)=\sum_{\ell=0}^{r}\ell-(d_{Q_i}^{+}(x^i)+1)$. Since $0\leq d_{Q_i}^+(y_j^i)\leq r $, we have 
\begin{align}\label{dy-xq}
\{d_{Q_i}^+(y_j^i):0\leq j\leq r-1\}
        =
        \{0,1,\ldots,r\}\backslash \{d_{Q_i}^+(x^i)+1\}.
\end{align}

Observe that
the value \(d_X^+(x^i)+d_{Q_i}^+(x^i)+2\) belongs to this set.
If this value occurs at some \(y_j^i\) adjacent to \(x^i\), then
\[d^+(y_j^i)
=k'-2+d_X^+(x^i)+d_{Q_i}^+(x^i)+2=d^+(x^i),
\]
contradicting  properness. Hence this value must occur at
the unique vertex \(y_0^i\) not adjacent to \(x^i\).

By  (\ref{dy-xq}),  $Q_i+(x^i,y_0^i)$ is a transitive tournament on \(r+1\) vertices.  But the added arc $(x^i,y_0^i)
$ goes from a vertex
of out-degree \(d^+_{Q_i}(x^i)+1\) to a vertex of larger out-degree \(d^+_{X}(x^i)+d^+_{Q_i}(x^i)+2\).
 This is impossible in a transitive tournament. Therefore
\(d_X^+(x^i)+d_{Q_i}^+(x^i)\geq r-1.
\)
\end{proof}
Since
\(
        d_{Q_i}^+(x^i)\leq r-1,
\)  \(
        d_{X}^+(x^i)\leq b-a,
\) and $k'=k+a-r+1$, by Claim \ref{X+Q} and 
equality~\eqref{dxq}, we obtain
\[
k+a\leq d^+(x^i)
        \leq
        k+b.
\]
Since \(X\) is a clique,  \(\{d^+(x^1),\ldots,d^+(x^{b-a+1})\}=\{k+a,\ldots,k+b\}.
\)
Hence, by
properness, no vertex of \(A\) can have an out-degree in
\([k+a,k+b].
\)
Thus \(F_2\) is an $r$-partite \([k+a,k+b]\)-forcer for \(A\) with upper
bound \(M\).
\end{proof}
\begin{figure}[htbp]
\centering
\begin{tikzpicture}[
    x=1cm,y=1cm,
    every node/.style={font=\small},
    line cap=round,
    line join=round,
    edge/.style={black, line width=0.55pt},
    missing/.style={black, dashed, line width=0.45pt},
    set/.style={draw, ellipse, minimum width=1.65cm, minimum height=1.05cm},
    copybox/.style={draw, rectangle, minimum width=1.15cm, minimum height=0.55cm},
    blockbox/.style={draw, rectangle, rounded corners=2pt,
                     minimum width=6.3cm, minimum height=5.6cm},
    qbox/.style={draw, ellipse, minimum width=1.65cm, minimum height=1.05cm},
    bigqbox/.style={draw, ellipse, minimum width=3.35cm, minimum height=3.10cm}
]

\node[copybox] (P1) at (-0.9,2.75) {\(P_1\)};
\node at (0.75,2.75) {\(\cdots\)};
\node[copybox] (Ptminus) at (2.45,2.75) {\(P_{t-1}\)};

\node[copybox] (Ptplus) at (-0.9,-3.9) {\(P_{t+1}\)};
\node at (0.75,-3.9) {\(\cdots\)};
\node[copybox] (PT) at (2.45,-3.9) {\(P_T\)};

\node[blockbox] (Pt) at (1.0,-0.6) {};
\node at (3.4,1.7) {\(P_t\)};

\node[set] (Y0) at (-3.45,1.35) {\(Y_0\)};
\node[set] (Y1) at (-3.45,0.20) {\(Y_1\)};
\node at (-3.85,-0.95) {\(\vdots\)};
\node[set] (Yr) at (-3.45,-2.25) {\(Y_{r-1}\)};

\node[draw, ellipse, minimum width=1.9cm, minimum height=2.2cm] (A) at (5.55,-0.6) {\(A\)};

\node[qbox] (Q1) at (0.15,1.30) {};
\node at (-0.18,1.30) {\(Q_1\)};

\coordinate (xone) at (0.6,1.42);
\fill (xone) circle (1.35pt);
\node[left=-2pt] at (xone) {\(x^1\)};

\node[rotate=90] at (0.25,0.50) {\(\cdots\)};

\node[bigqbox] (Qi) at (0,-1.35) {};
\node at (0.40,-0.20) {\(Q_i\)};

\coordinate (yi0) at (-0.42,-0.45);
\coordinate (yi1) at (-0.82,-0.78);
\coordinate (yid) at (-0.76,-1.44);
\coordinate (yir) at (-0.40,-1.86);
\coordinate (xi)  at (0.95,-1.12);

\fill (yi0) circle (1.25pt);
\fill (yi1) circle (1.25pt);
\fill (yir) circle (1.25pt);
\fill (xi)  circle (1.35pt);

\node[left=-13pt,yshift=8pt] at (yi0) {\(y_0^i\)};
\node[left=-5pt,yshift=-8pt] at (yi1) {\(y_1^i\)};
\node[left=-21pt,yshift=-7pt] at (yir) {\(y_{r-1}^i\)};
\node[rotate=24] at (yid) {\(\vdots\)};
\node[above right,xshift=1pt,yshift=-1pt] at (xi) {\(x^i\)};

\draw[missing] (xi) -- (yi0);
\draw[edge] (xi) -- (yi1);
\draw[edge] (xi) -- (yir);
\draw[edge] (yi0) -- (yi1);
\draw[edge] (yi1) -- (yir);
\draw[edge] (yi0) -- (yir);

\node at (1.55,-2.5) {\(\cdots\)};

\node[qbox, minimum width=1.95cm, minimum height=1.10cm] (Qm) at (2.90,-2.50) {};
\node at (2.83,-2.74) {\(Q_{b-a+1}\)};

\coordinate (xm) at (2.60,-2.13);
\fill (xm) circle (1.35pt);
\node[left=-34.5pt,yshift=-4.7pt] at (xm) {\(x^{b-a+1}\)};

\draw[edge] (xone) -- (xi);
\draw[edge] (xi) -- (xm);
\draw[edge] (xone) -- (xm);

\draw[edge]
    ($(Y0.east)+(-0.1,0.24)$)
    to[out=-20,in=145]
    ($(yi0)+(-0.0,0.005)$);

\draw[edge]
    ($(Y0.east)+(-0.05,-0.16)$)
    to[out=-18,in=160]
    ($(yi0)+(-0,-0.005)$);

\draw[edge]
    ($(Y1.east)+(-0.1,0.24)$)
    to[out=-20,in=158]
    ($(yi1)+(-0.,0.005)$);

\draw[edge]
    ($(Y1.east)+(-0.09,-0.23)$)
    to[out=-18,in=165]
    ($(yi1)+(-0.03,-0.005)$);

\draw[edge]
    ($(Yr.east)+(-0.1,0.24)$)
    to[out=5,in=185]
    ($(yir)+(-0.0,0.005)$);

\draw[edge]
    ($(Yr.east)+(-0.05,-0.18)$)
    to[out=20,in=200]
    ($(yir)+(-0.0,-0.005)$);

\draw[edge]
    ($(xone)+(0,0.005)$)
    to[out=2,in=180]
    ($(A.west)+(0.53,0.99)$);

\draw[edge]
    ($(xone)+(0,-0.005)$)
    to[out=-4,in=180]
    ($(A.west)+(0.19,0.65)$);

\draw[edge]
    ($(xi)+(0.05,0.005)$)
    to[out=0,in=180]
    ($(A.west)+(0.02,0.18)$);

\draw[edge]
    ($(xi)+(0.05,-0.005)$)
    to[out=-8,in=180]
    ($(A.west)+(0.03,-0.25)$);

\draw[edge]
    ($(xm)+(0.0,0.005)$)
    to[out=20,in=200]
    ($(A.west)+(0.25,-0.74)$);

\draw[edge]
    ($(xm)+(0.0,-0.005)$)
    to[out=10,in=200]
    ($(A.west)+(0.62,-1.03)$);
\end{tikzpicture}
\caption{A schematic picture of the \([k+a,k+b]\)-forcer \(F_2\). The rectangle shows one copy \(P_t\); the other copies \(P_1,\ldots,P_T\) are indicated outside.}
\label{fig:F2-forcer}
\end{figure}

\subsection{Controlling  $\text{Mad}(G)$} Chen, Mohar and Wu \cite{CM} controlled the maximum average degree of their forcers via
\(k\)-degeneracy. However, our forcers are no longer \(k\)-degenerate. In this
subsection, we develop a more general method (inspired by Hakimi's Lemma \ref{l1}) to control the maximum average
degree.

\begin{lem}\label{lemma:mad-charging}
Let \(G\) be a graph, and let \(U\subseteq V(G)\). Suppose that:
 (a)
\(\text{Mad}(G[U])\leq 2k;
\)
(b)
\(G- E(G[U])\) admits an orientation such that every vertex in $U$ has zero out-degree, and every vertex outside \( U\) has
 out-degree at most \(k\).
Then
\(\text{Mad}(G)\leq 2k.
\)
\end{lem}

\begin{proof}
Let \(H\subseteq G\) be any subgraph. It suffices to prove that 
\(e(H)\leq k|V(H)|.
\)
By condition (a), we have
\(e(H[V(H)\cap U])\leq k|V(H)\cap U|.
\)
 By condition (b), there is an orientation of $H-E(H[V(H)\cap U])$ such that every vertex in $V(H)\cap U$ has zero out-degree and every vertex outside $V(H)\cap U$ has out-degree at most $k$. Then
        $e(H)-e(H[V(H)\cap U])\leq k|V(H)\backslash U|$.
Therefore
\(e(H)\leq k|V(H)\cap U|+k|V(H)\backslash U|=k|V(H)|.
\)
\end{proof}

We now show that the forcers \(F_1\) and \(F_2\) constructed above satisfy
the condition (b)  in Lemma~\ref{lemma:mad-charging}.
\begin{lem}\label{lemma:forcer-charging}
Let \(r\geq 3\), and let \(a,b\) be integers with
\(1\le a\le b\) and \(b-a+1\le r-1\). Suppose that
\(k'=k+a-r+1\ge 2,
\)
        \(k'+\left\lceil\frac{b-a}{2}\right\rceil\le k\),
\(
        k'-1+\left\lceil\frac{r-1}{2}\right\rceil\le k
\) and $M\geq k+b$.
Let \(A\) be an independent set of size $k$ and let \(A'\subseteq A\) be a \(k'\)-set in some graph $G$.
Then the following statements  hold.

\smallskip
\noindent{\rm (i)}
Let \(F_1\) be the \([1,k+1]\)-forcer for \(A\) with upper bound $M$ constructed in
Lemma~\ref{lemma:low-forcer}. Then \(F_1\) admits an orientation such
that every vertex in $A$ has zero out-degree  and every vertex
outside \( A\) has out-degree at most \(k\).

\smallskip
\noindent{\rm (ii)}
Let \(F_2\) be the \([k+a,k+b]\)-forcer for \(A'\)  with upper bound $M$ constructed in
Lemma~\ref{lemma:high-forcer}. Then \(F_2\) admits an orientation in
which every vertex in $A'$ has zero out-degree  and every vertex
outside \( A'\) has out-degree at most \(k\).
\end{lem}

\begin{proof}

 (i) In $F_1$, let \(B\to A\), \(C\to A\) and
\(D\to C\).  It is easy to check that every vertex in \(A\)
has out-degree \(0\) and  every vertex in \(B\cup C\cup D\) has
out-degree at most \(k\). 

(ii) Note that all the edges of $F_2$ are incident with some vertex of a forcing block $P_t$. In each forcing
block $P_t$ of $F_2$, let
\(X=\{x^1,x^2,\ldots,x^{b-a+1}\}.
\)
Orient \(x^i\to A'\) and  the clique on \(X\) as a tournament with maximum out-degree at
most
\(\left\lceil(b-a)/2\right\rceil.
\)
For every \(i\) and \(j\), orient 
\(y_j^i\to x^i\) whenever such an edge exists,  \(y_j^i\to Y_j\),   and  orient the clique
\(\{y_0^i,y_1^i,\ldots,y_{r-1}^i\}
\)
as a tournament with maximum out-degree at most
\(\left\lceil(r-1)/2\right\rceil.
\)

Under this orientation, every \(x^i\) has  \(k'\) out-edges to \(A'\), and
 at most
\(\left\lceil(b-a)/2\right\rceil
\)
out-edges to \(X\). Hence
\[
        d^+(x^i)
        \leq
        k'+\left\lceil\frac{b-a}{2}\right\rceil
        \leq k.
\]
Since  \(y_j^i\) has at most one out-edge
to \(x^i\), \(k'-2\) out-edges to \(Y_j\),   and  at most
\(\left\lceil(r-1)/2\right\rceil
\)
out-edges to
\(\{y_0^i,y_1^i,\ldots,y_{r-1}^i\}.
\)
Therefore
\[
        d^+(y_j^i)
        \leq
        (k'-2)+1+\left\lceil\frac{r-1}{2}\right\rceil
        =
        k'-1+\left\lceil\frac{r-1}{2}\right\rceil
        \leq k.
\]

Also, the vertices in \(A'\cup Y_0\cup\cdots\cup Y_{r-1}\) have out-degree zero. 
\end{proof}

\subsection{The construction}

Let \(r\geq 3\), $k>\frac{5r^2-10r+4}{2}$ and 
\(M=k+\left\lfloor \frac{5r}{2}\right\rfloor-3
\).
We first construct $G_0$.  Initially, $G_0=A$, where $A$ is an independent set  of size \(k\). Next, we attach to $A$ the following forcers.

(1) Attach to \(A\) a \([1,k+1]\)-forcer \(F_1\) for $A$ with upper bound \(M\) as in
Lemma~\ref{lemma:low-forcer}.

(2) Choose a subset
\(A^-\subseteq A
\)
of size
\(k'=:k-r+3.
\) Applying Lemma \ref{lemma:high-forcer} with $(a,b,A)=(2,r-1,A^-)$, 
attach to \(A^-\) a \([k+2,k+r-1]\)-forcer $F_2$ for $A^-$ with upper bound \(M\).

(3) If \(r\geq 4\), choose another subset \(A^*\) with 
\(A^-\subseteq A^*\subseteq A
\) of size  \(|A^*|=k-\left\lceil \frac{r}{2}\right\rceil+1\).  Applying Lemma \ref{lemma:high-forcer} with $(a,b,A)=(\lfloor \frac{r}{2}\rfloor,\lfloor \frac{3r}{2}\rfloor-2,A^*)$, attach to \(A^*\) a
\([k+\lfloor \frac{r}{2}\rfloor,k+\lfloor \frac{3r}{2}\rfloor-2]\text{-forcer}
\) $F_3$
with upper bound \(M\).

When \(r=3\), since \(r-1=\lfloor \frac{3r}{2}\rfloor-2=2\),  we do not use \(F_3\) to attach to $A$. 
That is,
\(
        G_0=F_1\cup F_2
\)
if \(r=3\) and
\(
        G_0=F_1\cup F_2\cup F_3
\)
if \(r\geq 4\). By construction, \(G_0\) is
\(r\)-partite.

Now we choose a subset
\(
        U\subseteq A^-
\)
of size
\(
        \left\lfloor \frac{2k}{r-1}\right\rfloor .
\)
For \(r=3\), we can see that \(|U|=|A^-|=|A|=k\).
For \(r\ge4\), this is possible since
\(
        \left\lfloor \frac{2k}{r-1}\right\rfloor \le k-r+3=|A^-|.
\)
Take \(r\) disjoint copies of \(G_0\), denoted by
\(G_1,G_2,\ldots,G_r.
\)
For each \(1\leq i\leq r\), denote the corresponding sets in \(G_i\) by
\(A^i, A^{i-}, A^{i*}, U^i,
\)
where \(A^{i*}\) is used only when \(r\geq 4\). Add all edges between
\(U^i\) and \(U^j\) for every \(i\neq j\). Denote the resulting graph by
\(G\), and put
\(U^*=U^1\cup\cdots\cup U^r .
\)
Then \(G[U^*]\) is a complete \(r\)-partite graph, which implies that \(\chi(G)=r\).

\begin{lem}\label{lemma:five-halves-construction}
For every  \(r\geq 3\) and  \(k>\frac{5r^2-10r+4}{2}\), the $r$-partite graph
\(G\) constructed above satisfies
\(\left\lceil \frac{\text{Mad}(G)}2\right\rceil=k
\)
and
\[
\vec{\chi}(G)\geq
k+\left\lfloor \frac{5r}{2}\right\rfloor-2.
\]
\end{lem}
\begin{proof}
 We first prove that
\(\left\lceil \frac{\text{Mad}(G)}2\right\rceil=k
\) by using Lemma \ref{lemma:mad-charging} with $(G,U)=(G,U^*)$.
Since \(G[U^*]\) is a complete \(r\)-partite graph with each part
of size \(\lfloor \frac{2k}{r-1}\rfloor\), it is easy to verify that $\text{Mad}(G[U^*])\leq 2k$, satisfying condition \({\rm (a)}\) of Lemma~\ref{lemma:mad-charging}.
Next we verify condition \({\rm (b)}\) of Lemma \ref{lemma:mad-charging}, that is, we only need to consider the subgraph $G-E(G[U^*])$. Note that \(G-E(G[U^*])\) is the union of the forcers attached in the
\(r\) copies.  Within each copy, these forcers may intersect only on the
corresponding set \(A^i\), and the selected set \(U^i\) is contained in
\(A^i\). And $U^*\ss \cup_{i=1}^r A^i$, where $A^i$ is a copy of $A$. It is easy to verify that all $F_1,F_2,F_3$ satisfy the conditions of Lemma~\ref{lemma:forcer-charging}.
By
Lemma~\ref{lemma:forcer-charging},
all edges of these forcers can be oriented so that every vertex in  \(\cup_{i=1}^r A^i\) has zero out-degree and every vertex outside \(\cup_{i=1}^r A^i\) has 
out-degree at most \(k\).
 By Lemma~\ref{lemma:mad-charging},
\(\text{Mad}(G)\leq 2k.
\)

On the other hand, in the \([1,k+1]\)-forcer, the subgraph induced by
\(A\) and \(C\) contains a complete bipartite graph with parts of sizes
\( k\) and \(k(kM+1).
\)
Its average degree is
\[
        \frac{2k\cdot k(kM+1)}{k+k(kM+1)}
        =2k-
        \frac{2k}{kM+2}.
\]
 Since $M=k+\left\lfloor 5r/2\right\rfloor-3$, this gives
\(\left\lceil \frac{\text{Mad}(G)}2\right\rceil=k.
\)
\medskip

Now
we prove the lower bound on \(\vec{\chi}(G)\). Suppose, for
a contradiction, that \(G\) has a proper orientation  such that
\(d^+(v)\leq M\) for every
       \(v\in V(G).
\)

Fix \(i\in\{1,\ldots,r\}\). Since \(F_1\) is a \([1,k+1]\)-forcer for
\(A^i\), every vertex of \(A^i\) avoids the out-degree values
\([1,k+1].
\)
Since \(F_2\) is a \([k+2,k+r-1]\)-forcer for \(A^{i-}\), every vertex
of \(A^{i-}\) avoids
\([k+2,k+r-1].
\)
If \(r\geq 4\), then \(F_3\) is a
\([k+\lfloor \frac{r}{2}\rfloor,k+\lfloor \frac{3r}{2}\rfloor-2]\text{-forcer}
\)
for \(A^{i*}\). Since
\(A^{i-}\subseteq A^{i*},
\)
every vertex of \(A^{i-}\) also avoids
\([k+\lfloor \frac{r}{2}\rfloor,k+\lfloor \frac{3r}{2}\rfloor-2].
\)
Thus every vertex of \(A^{i-}\), and hence every vertex of \(U^i\),
satisfies
\[
        d^+(v)\notin [1,k+\lfloor \frac{3r}{2}\rfloor-2].
\]
When \(r=3\), this conclusion follows from \(F_1\) and \(F_2\), since
\(r-1=\lfloor \frac{3r}{2}\rfloor-2\).

Since
\( d^+(v)\leq M=k+\lfloor \frac{5r}{2}\rfloor-3,
\)
every vertex in \(U^i\) has out-degree in
\(\{0\}\cup[k+\lfloor \frac{3r}{2}\rfloor-1,k+\lfloor \frac{5r}{2}\rfloor-3]\), which has exactly \(r\) values.
Now consider the complete \(r\)-partite graph \(G[U^*]\) with
parts
\(U^1,U^2,\ldots,U^r.
\)  It is easy to see that 
 $\{d^+(u_1),\ldots,d^+(u_r)\}=\{0\}\cup[k+\lfloor \frac{3r}{2}\rfloor-1,k+\lfloor \frac{5r}{2}\rfloor-3]$ for any $u_i\in U^i$.
Consequently,
\[
\binom r2 \lfloor \frac{2k}{r-1}\rfloor^2=e(G[U^*])\le \sum_{v\in U^*}d^+(v) =
\lfloor \frac{2k}{r-1}\rfloor\sum_{j=1}^{r-1}(k+\lfloor \frac{3r}{2}\rfloor-2+j)\]

\[=\lfloor \frac{2k}{r-1}\rfloor(r-1)\left(k+\left\lfloor \frac{3r}{2}\right\rfloor+\frac r2-2\right).
\]
Since $k>\frac{5r^2-10r+4}{2}$, a direct calculation shows  the above inequality is false. Hence
\(\vec{\chi}(G)\geq M+1
=k+\left\lfloor\frac{5r}{2}\right\rfloor-2.
\)
\end{proof}

\section{Proof of Theorem \ref{t1}}

Before proving it, we introduce two lemmas  as below.

\begin{lem}
  \label{l1}  (Hakimi \cite{H}) A graph $G$ admits an orientation such that the maximum out-degree is at most $k$ if and only if $\text{Mad}(G)  \leq 2k$.
\end{lem}

\begin{lem}
    \label{l2}(Ore \cite{Ore}, \cite{Ore1957}) Let $G=(U\cup V, E)$ be a bipartite graph and  $w: U\cup V \rightarrow \mathbb{N}_{0}$ be a weight function on  the vertices of $G$ such that  $w(u)=1$ for any $u\in U$ and $w(v)\geq 1$ for any $v\in V$.  If
      $\sum_{u\in S}w(u)\leq w(N(S))$ for any $S\subseteq U$,
then there is a subgraph $M$ of $G$ such that  $d_M(u)=1$ for  each $u\in U$, and $d_M(v)\leq w(v)$ for each $v\in V$.
\end{lem}

Now let $G$ be a 3-partite graph with partition $V_1,V_2,V_3$, and each $V_i$ is independent. 
Let $\left\lceil \frac{\text{Mad}(G)}{2} \right\rceil = k $.
Let $D_0$ be an orientation given by Lemma \ref{l1}. Then $d_{D_0}^+(v)\leq k$ for any $v\in V(D_0)$.    

We first define some  notions.
We say $v$ is an \emph{oriented} vertex if every edge incident with $v$ has been oriented. Otherwise, $v$ is \emph{unoriented}.
The \emph{potential out-degree} of a vertex \(v\), denoted by \(d_p(v)\),
is the number of unoriented edges incident with \(v\) plus the number of
out-edges incident with \(v\).  Let  $d_p^+(v)$ be the sum of the number of out-edges incident with $v$. Then $d_p^+(v)\leq d_p(v)\leq d(v)$.
If $v$ is an oriented vertex, then $d_p(v)=d_p^+(v)$. 

We will construct a proper orientation step by step.  
 We first construct some independent sets $A_{k+i}$ for $i\in \{2,3,4,5,6,7\}$ and the orientations of edges incident exactly with vertices of $A_{k+i}$ such that every $v\in A_{k+i}$ becomes oriented and $d_p^+(v)=k+i$.
Note that for any $v\in V(G)$,  $d_p(v)=d(v)$ Initially, and $d_p(v)$ will decrease as we give the orientations step by step.

\vskip 1mm

\noindent{\bf{Step 1.}}
We carry out the following for $i=7,6,5$ one by one:
Let $U_{k+i}:=\{v\in V(G)\backslash (\cup_{j=i+1}^{7} A_{k+j}):d_p(v)\geq k+i\}$ (If $i=7$, then $\cup_{j=i+1}^{7} A_{k+j}=\emptyset$). We choose an independent set $A_{k+i}\subseteq U_{k+i}$ such that 
$A_{k+i} \cap  V_{8-i}$ is as large as possible. Subject to this,  $A_{k+i}$ is  as large as possible. 
Then $d_p(v)\leq k+i-1$ for any $v\in V_{8-i}\backslash (\cup_{j=i}^7A_{k+j})$. Note that for any $v\in A_{k+i}$, all the  out-edges incident with $v$ (if exist) are also
the arcs in $D_0$.

We orient the unoriented edges incident with  $v\in A_{k+i}$ out of $v$ if it is also out of $v$ in $D_0$. 
Since $v\in A_{k+i}$ has at most $k$ out-edges in $D_0$,  $d_p^+(v)\leq k$.
Then we  orient the remaining unoriented edges incident with  $v\in A_{k+i}$ such that $v$ becomes oriented and  $d_p^+(v)=k+i$. 
Let $uv$ be an edge such that 
 $u\notin \cup_{j=i}^7A_{k+j}$, $v\in \cup_{j=i}^7A_{k+j}$ and $uv$ has been oriented from $u$ to $v$. Since $(u,v)$ is an in-edge of $v$, by the above orientations, $(u,v)$ must be an arc in $D_0$. Then
all the out-edges incident with $u\notin \cup_{j=i}^7A_{k+j}$ are also the arcs in $D_0$. 
And $d_p^+(u)\leq k$ for any $u\notin \cup_{j=i}^7A_{k+j}$. We return to carry out the next $i$ until $i=5$.

By Step 1, we can conclude the followings.

(1.1) All the vertices in $\cup_{i=5}^7A_{k+i}$ have been oriented;
 
 (1.2)  $d_p^+(v)=k+i$ for any $v\in A_{k+i}$ and $i\in \{5,6,7\}$;

 (1.3) $d_p(v)\leq k+7-j$ for any $v\in V_j\backslash (\cup_{i=5}^7A_{k+i})$ and $j\in \{1,2,3\}$;

 (1.4) All the  out-edges incident with $v\notin \cup_{i=5}^7A_{k+i}$ are also the arcs in $D_0$;
 
 (1.5) $d_p^+(v)\leq k$ for any $v\notin \cup_{i=5}^7A_{k+i}$;

 (1.6) $uv$ has not been oriented for any $uv\in E(G)$ with  $u,v\notin \cup_{i=5}^7A_{k+i}$.

 \vskip 1mm
The next two steps are quite different and need to orient the edges   carefully.

\noindent{\bf{Step 2.}} Let $U_{k+4}$ be a vertex set  containing   unoriented vertices in $V_1$ with potential out-degree at least $k+4$, and  vertices in $V_2\cup V_3$ with potential out-degree equal to $k+4$. 
We define a weight function $w$ on $U_{k+4}$:
 \vskip 2mm
      \hspace{10mm}$w(v)=
\begin{cases}
d_p(v)-k-4 & v\in V_1\cap U_{k+4} \\
 1 & v\in (V_2\cup V_3)\cap U_{k+4}\\
\end{cases}$
  \vskip 2mm

Now we choose an independent set $A_{k+4}\subseteq U_{k+4}$ with maximum sum of weights. Subject to this,
 $A_{k+4}\cap V_1$ is as large as possible. Subject to these two conditions,
 $A_{k+4}$ is as large as possible. Let $X_{k+4}=U_{k+4} \backslash A_{k+4}$. Then  $d_p(v)\leq k+3$ for any $v\in V_1\backslash (\cup_{i=4}^7 A_{k+i}\cup X_{k+4})$.
For any $v\in X_{k+4}\cap V_1$,   by our choice of $A_{k+4}$, $|N(v)\cap A_{k+4}\cap (V_2\cup V_3)|\geq w(v)+1$.

(O.2.1) For any $v\in A_{k+4}\cap (V_2\cup V_3)$,   we orient all unoriented edges incident with $v$ out of $v$.

Then all vertices in $A_{k+4}\cap (V_2\cup V_3)$ have been oriented. For any $v\in A_{k+4}\cap (V_2\cup V_3)$, since $d_p(v)=k+4$,  $d_p^+(v)=k+4$. 
 Since $|N(v)\cap A_{k+4}\cap (V_2\cup V_3)|\geq w(v)+1$ for any $v\in X_{k+4}\cap V_1$,
 $d_p(v)$ becomes at most $ k+3$. Hence,  $d_p(v)\leq k+3$ for any $v\in  V_1\backslash (\cup_{i=4}^7A_{k+i})$.

 Let $X_{k+4}'$ be the set of vertices in $X_{k+4}\cap   V_3$ whose potential out-degree is still equal  $k+4$ after the  orientation (O.2.1). Observe that no vertex in $X_{k+4}'$ has a neighbor in $A_{k+4}\cap V_2$; otherwise, after the orientation (O.2.1), its potential out-degree would decrease to at most $k+3$, contradicting the definition of $X_{k+4}'$. By our choice of $A_{k+4}$, for any     $S\subseteq X_{k+4}'$,  
       $w(S)\leq w(N(S)\cap A_{k+4}\cap V_1)$.
Applying  Lemma \ref{l2} on $(U,V)=(X_{k+4}',A_{k+4}\cap V_1)$, there exists a subgraph $M\ss G[X_{k+4}', A_{k+4}\cap V_1]$ such that    $d_{M}(u)=1$ for  each $u\in X_{k+4}'$, and   $d_{M}(v)\leq w(v)\leq 2$ for each $v\in A_{k+4}\cap V_1$. 

(O.2.2) We orient the edges in $M$ from $A_{k+4}\cap V_1$ to $X_{k+4}'$. Orient the unoriented edges incident with $v\in A_{k+4}\cap V_1$ out of $v$ which are also out of $v$ in $D_0$. 

After the orientation (O.2.2),  $d_p(v)\leq k+3$ for any $v\in  V_3\backslash (\cup_{i=4}^7A_{k+i})$.
By (1.4) and the orientation (O.2.2), we can check that  $d_p^+(v)\leq k+2< d_p(v)$ for each $v\in A_{k+4}\cap V_1$.

(O.2.3) We  orient the unoriented edges incident with $v\in A_{k+4}\cap V_1$ such that $v$ becomes oriented and $d_p^+(v)= k+4$. 

By the above orientations, we can conclude the followings. 

(2.1) All the vertices in $\cup_{i=4}^7A_{k+i}$ have been oriented;
 
 (2.2)  $d_p^+(v)=k+i$ for any $v\in A_{k+i}$ and $i\in \{4,5,6,7\}$;

 (2.3) $d_p(v)\leq k+3$ for any $v\in V_1\cup V_3\backslash (\cup_{i=4}^7A_{k+i})$ and $d_p(v)\leq k+5$ for any $v\in V_2\backslash (\cup_{i=4}^7A_{k+i})$;

 (2.4) All the  out-edges incident with $v\notin \cup_{i=4}^7A_{k+i}$ are also the arcs in $D_0$;

 (2.5) $d_p^+(v)\leq k$ for any $v\notin \cup_{i=4}^7A_{k+i}$;

 (2.6) $uv$ has not been oriented for any $uv\in E(G)$ with  $u,v\notin \cup_{i=4}^7A_{k+i}$.

 \vskip 1mm

\noindent{\bf{Step 3.}}  We carry out the following orientations for $i=3,2$ one by one.
 Let $U_{k+i}:=\{v\in V(G)\backslash\cup_{j=i+1}^{7}A_{k+j} :d_p(v)\geq k+i\}$. We choose an independent set $A_{k+i}\subseteq U_{k+i}$ such that  $A_{k+i}$ is as large as possible. Subject to this, $A_{k+i}\cap V_2$  is as large as possible. Let $X_{k+i}=U_{k+i}\backslash A_{k+i}$. 
Then $d_p(v)\leq k+i-1$ for any $v\in V_2\backslash (\cup_{j=i}^7 A_{k+j}\cup X_{k+i})$.
 For any $v\in X_{k+i}\cap V_2$,   by our choice of $A_{k+i}$, $|N(v)\cap A_{k+i}\cap (V_1\cup V_3)|\geq 2$. 
 
 (O.3.1)
For any $v\in A_{k+i}\cap (V_1\cup V_3)$,   we orient all unoriented edges incident with $v$ out of $v$. 

Then all vertices in $A_{k+i}\cap (V_1\cup V_3)$ have been oriented.  Since $d_p(v)=k+i$ for any $v\in A_{k+i}\cap (V_1\cup V_3)$, $d_p^+(v)=k+i$. Since $|N(v)\cap A_{k+i}\cap (V_1\cup V_3)|\geq 2$
for any $v\in X_{k+i}\cap V_2$, by (O.3.1),
$d_p(v)$ becomes at most $ k+2i-3$.
Then $d_p(v)$ becomes at most $ k+2i-3$ for any $v\in V_2\backslash \cup_{j=i}^7 A_{k+j}$.

 For any $j\in \{1,3\}$, let $X_{k+i,j}$ be the set of vertices in $X_{k+i}\cap   V_j$ whose potential out-degree is still equal $k+i$ after the  orientation (O.3.1). Observe that any vertex in $X_{k+i,j}$ has no neighbor in $A_{k+i}\cap V_{4-j}$ (otherwise, its potential out-degree decreases to at most $k+i-1$). By our choice of $A_{k+i}$,  for any     $S\subseteq X_{k+i,j}$, we have
       $|S|\leq |N(S)\cap A_{k+i}\cap V_2|$.
Applying Lemma \ref{l2} to $(U,V)=(X_{k+i,j},A_{k+i}\cap V_2)$ and $w(v)=1$ for any $v\in U\cup V$, there exists a matching $M_j\ss G[X_{k+i,j}, A_{k+i}\cap V_2]$ such that    $d_{M_j}(u)=1$ for  each $u\in X_{k+i,j}$, and   $d_{M_j}(v)\leq 1$ for each $v\in A_{k+i}\cap V_2$.
 
 (O.3.2) For any $j\in \{1,3\}$, we orient the edges in $M_j$ from $A_{k+i}\cap V_2$ to $X_{k+i,j}$. Orient the unoriented edges incident with $v\in A_{k+i}\cap V_2$ out of $v$ which are also out of $v$ in $D_0$. 
 
 Then $d_p(v)\leq k+i-1$ for any $v\in V_1\cup V_3\backslash(\cup_{j=i}^7A_{k+j})$.
 By (2.4) and the orientation (O.3.2),  we can check that  $d_p^+(v)\leq k+2\leq  d_p(v)$ for each $v\in A_{k+i}\cap V_2$. 
 
 (O.3.3) We  orient the unoriented edges incident with $v\in A_{k+i}\cap V_2$ such that $v$ becomes oriented and $d_p^+(v)= k+i$.

By the above orientations, we can conclude the followings.

 (3.1) All the vertices in $\cup_{j=i}^7A_{k+j}$ have been oriented;
 
 (3.2)  $d_p^+(v)=k+j$ for any $v\in A_{k+j}$ and $j\in \{i,i+1,\ldots,7\}$;

 (3.3) $d_p(v)\leq k+i-1$ for any $v\notin  \cup_{j=i}^7A_{k+j}$, except that when $i=3$ and $v\in V_2\backslash (\cup_{j=3}^7A_{k+j})$, $d_p(v)\leq k+3$;
 
  (3.4) All the  out-edges incident with $v\notin \cup_{j=i}^7A_{k+j}$ are also the arcs in $D_0$;

 (3.5) $d_p^+(v)\leq k$ for any $v\notin \cup_{j=i}^7A_{k+j}$;

(3.6) Let $uv$ be an edge with $u,v\notin \cup_{j=i}^7A_{k+j}$.
Then $uv$ has not been oriented.

 We return to carry out the above arguments for the case when $i=2$ in this step.
 
\vskip 1mm

\vskip 1mm

By (3.1)-(3.6), we are ready to  orient all remaining unoriented edges as below.
In each step, we choose an unoriented vertex $v$ with the maximum $d_p(v)$ and then orient all the edges incident with $v$ out of $v$. 
By all the orientations above, we can check that we obtain a proper orientation as required.
  \bqed

\vskip 5mm

\noindent{\bf\large Acknowledgements}
\vskip 2mm

The authors sincerely thank Xiaolan Hu from Central China Normal University.  This research is supported by 
 National Key R\&D Program of China under grant number 2023YFA1010202 and NSFC under grant number 12401447.
   Lanchao Wang is supported by the Institute for Basic Science (IBS-R029-C4) and  by National Key R\&D Program of China under grant number 2024YFA1013900, NSFC under grant number 12471327 and by China Scholarship Council.

\end{document}